\documentclass{amsart}

\usepackage{bbding}
\usepackage{amsmath}
\usepackage{amssymb,amsfonts,amsthm}
\usepackage[all]{xy}

\newtheorem{theorem}{Theorem}[section]
\newtheorem{lemma}[theorem]{Lemma}
\newtheorem{proposition}[theorem]{Proposition}
\newtheorem{corollary}[theorem]{Corollary}
\theoremstyle{definition}
\newtheorem{definition}[theorem]{Definition}
\newtheorem{example}[theorem]{Example}

\newtheorem{remark}[theorem]{Remark}
\numberwithin{equation}{section}

\newcommand{\blankbox}[2]

\begin{document}
\title{Extending structures and classifying complements for left-symmetric algebras}
\author{Yanyong Hong}
\address{College of Science, Zhejiang Agriculture and Forestry University,
Hangzhou, 311300, P.R.China}
\email{hongyanyong2008@yahoo.com}

\thanks{Project supported by the National Natural Science Foundation of China (No. 11501515), the Zhejiang Provincial Natural Science Foundation of China (No.LQ16A010011), and the Scientific Research Foundation of Zhejiang Agriculture and Forestry University (No.2013FR081)}
\begin{abstract}
 Let $A$ be a left-symmetric (resp. Novikov) algebra, $E$ be a vector space containing $A$ as a subspace and $V$ be a complement of
 $A$ in $E$.
 The extending structures problem which asks for the classification of all left-symmetric (resp. Novikov) algebra structures on
 $E$ such that $A$ is a subalgebra of $E$ is studied. In this paper, the definition of the unified product of left-symmetric (resp. Novikov) algebras is introduced. It is shown that there exists a left-symmetric (resp. Novikov) algebra structure on $E$ such that $A$ is a subalgebra of $E$ if and only if $E$ is isomorphic to a unified product of $A$ and $V$. Two cohomological type objects $\mathcal{H}_A^2(V,A)$ and  $\mathcal{H}^2(V,A)$ are constructed to give a theoretical answer to the extending structures problem. Furthermore, given an extension
 $A\subset E$ of
 left-symmetric (resp. Novikov) algebras, another cohomological type object is constructed to classify
 all complements of $A$ in $E$. Several special examples are provided in details.

\end{abstract}

\maketitle

\section{Introduction}
Left-symmetric algebras are a class of Lie-admissible algebras whose commutators are Lie algebras. They arose from the study of affine manifolds and affine structures on Lie
groups \cite{Ko}, convex homogeneous
cones \cite{V},  deformation of associative algebras \cite{G} and so on. Novikov
algebra is a left-symmetric algebra whose right multiplications are commutative. It was essentially stated in \cite{GD} that it corresponds to a certain Hamiltonian operator. Such an algebraic structure also appeared
in \cite{BN} from the point of view of Poisson structures of
hydrodynamic type. The name ``Novikov algebra" was given by Osborn
in \cite{Os}. In \cite{Bu1}, Burde gave a
survey about left-symmetric algebras which showed that they play important roles in many fields in mathematics and mathematical physics such as vector fields, rooted tree algebras, words in two letters, operad theory, vertex algebras, deformation complexes of algebras, convex homogeneous cones, affine manifolds, left-invariant affine structures on Lie groups (see \cite{Bu1} and the references therein).

In this paper, our first aim is to study the following question about left-symmetric (resp. Novikov) algebras:\\
\\
Let $A$ be a left-symmetric (resp. Novikov) algebra and $E$ a vector space containing $A$ as a subspace.
Describe and classify all left-symmetric (resp. Novikov) algebra structures on $E$ such that $A$ is a subalgebra of
$E$.\\
\\
This problem is called \emph{extending structures problem}. The extending structures problems for groups, associative algebras, Hopf algebras, Lie algebras, and Leibniz algebras have been studied in \cite{AM3, AM2, AM5, AM4, AM6} respectively. Since left-symmetric algebras and Novikov algebras are important algebraic structures in mathematics and mathematical physics, it is interesting and meaningful to study the extending structures problems for left-symmetric algebras and Novikov algebras. Certainly, this problem is very difficult. When $A=\{0\}$, this problem is equivalent to
classify all left-symmetric algebras and Novikov algebras on an arbitrary vector space $E$. Of course, when the dimension of $E$ is large,
it is impossible to classify all algebraic structures. In fact, the classifications of left-symmetric algebras in dimension 2 and dimension 3 over $\mathbb{C}$ were obtained in \cite{BM1,Bai1}, and the classifications of Novikov algebras in dimension 2 and dimension 3 over $\mathbb{R}$ and $\mathbb{C}$ were obtained in \cite{BM2,Dd}. But, as far as we know, the complete classifications of left-symmetric algebras (resp. Novikov algebras) whose dimension are larger than 3 have not been obtained so far. There are only some papers on the classification of some special left-symmetric algebras and Novikov algebras of dimensions 4 and 5, such as \cite{DIO,Ki1,Ki2,Dd}. Therefore, we will assume that $A\neq \{0\}$ below. In fact, the extending structures problem generalizes some classic algebraic problems. For example, the extension of the left-symmetric algebra $A$ by an abelian left-symmetric algebra $C$ is characterized by the second cohomology group $H^2(C,A)$ (see \cite{Dz});
the problem which asks for describing and classifying the extensions of the left-symmetric algebra $A$ by another left-symmetric algebra $B$ can be answered by bicrossed products of $A$ and $B$
(see \cite{Bai}) and this problem is called the \emph{factorization problem}. In this paper, our aim is to give a theoretical interpretation of the extending structures problem for left-symmetric (resp. Novikov) algebras using some cohomological type objects and provide some detailed answers to it in certain special cases which depend on the choice of $A$ and its codimension in $E$.

Our second aim is to study the following problem:\\
\\
Let $A\subseteq E$ be an extension of left-symmetric (resp. Novikov) algebras. If there exists a complement of $A$ in $E$, describe and classify all
complements of $A$ in $E$, i.e. subalgebras $B$ of $E$ such that $E=A+B$ and $A\cap B=\{0\}$.
\\
\\
This problem is called \emph{classifying complements problem}. The problems for groups, associative algebras, Hopf algebras, Lie algebras, and Leibniz algebras have been studied in \cite{AM7, A1, AM2, AM6} respectively. The classifying complements problem is the converse of the factorization problem. It is known from \cite{Bai} that if $B$ is an $A$-complement of $E$, then $E$ is isomorphic to a bicrossed product $A\bowtie B$ associated with a matched pair of $A$ and $B$. Then, this problem can reduce to the following question: describe and classify all left-symmetric (Novikov) subalgebras $C$ such that $A\bowtie B\cong A\bowtie C$. By studying this question, we show that given a complement $B$, any
$A$-complements of $E$ can be obtained from $B$ by some deformation map $\varphi: B\rightarrow A$.

This paper is organized as follows. In Section 2, some preliminaries about left-symmetric (resp. Novikov) algebras are recalled. In Section 3, we introduce the definition of unified product $A\natural V$ of left-symmetric (resp. Novikov) algebras associated with an extending datum $\Omega(A, V)=(l_A, r_A, l_V, r_V, f, \cdot)$. The sufficient and necessary condition to ensure that $A\natural V$ with a given canonical product is a left-symmetric (resp. Novikov) algebra is given. Then, we show that there exists a left-symmetric (resp. Novikov) algebra structure on $E$ such that $A$ is a subalgebra of $E$ if and only if $E$ is isomorphic to a unified product of $A$ and $V$ (see Theorem \ref{pt1}). According to this result, we only need to study the following question: when two unified products $A\natural V$ and $A\natural^{'} V$ corresponding to two extending datums $\Omega(A,V)=(l_A,r_A,l_V,r_V,f,\cdot)$ and $\Omega^{'}(A,V)=(l_A^{'},r_A^{'},l_V^{'},r_V^{'},f^{'},\cdot^{'})$ are isomorphic. This question is solved in Lemma \ref{pt2}. By Lemma \ref{pt2}, we construct two cohomological type objects to give a theoretical answer to the extending structures
problem for left-symmetric (resp. Novikov) algebras. The first one $\mathcal{H}_A^2(V,A)$ classifies all left-symmetric (resp. Novikov) algebras on
$E$ up to an isomorphism that stabilizes $A$, while the second object $\mathcal{H}^2(V,A)$ provides the classification of all extending structures
of $A$ to $E$ up to isomorphism which stabilizes $A$ and co-stabilizes $V$. In Section 4, a class of special extending structures of left-symmetric (resp. Novikov) algebras named \emph{flag extending structures} are studied in detail. Several examples for computing $\mathcal{H}_A^2(V,A)$ and
$\mathcal{H}^2(V,A)$ are given. In Section 5, the classifying complements problem for left-symmetric (resp. Novikov) algebras is studied. We show that if $A$ is a left-symmetric (resp. Novikov) algebra and $B$ is a complement of $A$ in $E$, then the isomorphism classes of all $A$-complements of $E$ can be parameterized by a cohomological object  $\mathcal{HC}^2(B,A|(l_A,r_A,l_B,r_B))$, where $l_A$, $r_A$, $l_B$, $r_B$ are defined in the proof of Theorem \ref{pt1}.

Throughout this paper, $k$ is a field. All vector spaces, left-symmetric algebras, Novikov algebras, Lie algebras, linear or bilinear maps are over $k$.

{\bf ACKNOWLEDGMENT} This work was done during the author's visit to TongJi University. He would like to thank Professor Yucai Su for the
hospitality, comments on this work and valuable discussions on Lie algebras.

\section{Preliminaries}
In this section, we will recall some basic definitions and some facts
about left-symmetric algebras and Novikov algebras.

\begin{definition}
A \emph{left-symmetric algebra} is a vector space $A$ together with a bilinear product $\circ: A\times A\rightarrow A$ satisfying
\begin{eqnarray}
(a,b,c)=(b,a,c),~~~~a,~b,~c\in A,\end{eqnarray}
where the associator $(a,b,c)=(a\circ b)\circ c-a\circ (b\circ c)$.

A \emph{Novikov algebra} $(A,\circ)$ is a
left-symmetric algebra with the operation $``\circ"$ satisfying:
\begin{equation}
(a\circ b)\circ c=(a\circ c)\circ b,~~~~~a,~b,~c\in A.
\end{equation}
\end{definition}

Left-symmetric algebras are Lie-admissible algebras (see \cite{Me}).
\begin{proposition}
Let $A$ be a left-symmetric algebra. The commutator
\begin{eqnarray}
[a,b]=a\circ b-b\circ a, ~~~~\forall~~a,~~b\in A,
\end{eqnarray}
defines a Lie algebra $\mathfrak{g}(A)$, which is called the sub-adjacent Lie algebra of $A$ and $A$ is also called the compatible
left-symmetric algebra structure on the Lie algebra $\mathfrak{g}(A)$.
\end{proposition}

Next, we introduce the definitions of bimodules of left-symmetric algebra and Novikov algebra.

\begin{definition}
Let $A$ be a left-symmetric algebra and $M$ be a vector space. Let $S$, $T: A\rightarrow gl(M)$ be two linear maps.
$M$ (or $(S,T,M)$ ) is called a \emph{bimodule} of $A$ if
\begin{eqnarray}
\label{e1}S(x)S(y)v-S(xy)v=S(y)S(x)v-S(yx)v,\\
\label{e2}S(x)T(y)v-T(y)S(x)v=T(xy)v-T(y)T(x)v,
\end{eqnarray}
for any $x$, $y\in A$ and $v\in M$.

A \emph{bimodule} $M$ of a Novikov algebra $(A,\circ)$ is a vector space with two linear maps $S$, $T: A\rightarrow gl(M)$ satisfying
(\ref{e1}), (\ref{e2}) and
\begin{eqnarray}
S(x\circ y)v=T(y)S(x)v,\\
T(x)T(y)v=T(y)T(x)v,
\end{eqnarray}
for any $x$, $y\in A$ and $v\in M$.
\end{definition}

The following fact is obvious.
\begin{proposition}
Let $(S, T, M)$ be a bimodule of a left-symmetric algebra $A$. Then,
$\rho=S-T$ is a representation of the Lie algebra $\mathfrak{g}(A)$.
\end{proposition}

Finally, in order to study the extending structures problem, we need to introduce the following definition.
\begin{definition}
Let $A$ be a left-symmetric (resp. Novikov) algebra, $E$ a vector space such that $A$ is a subspace
of $E$ and $V$ a complement of $A$ in $E$. For a linear map $\varphi: E\rightarrow E$, the following
diagram is considered:
$$\xymatrix{ {A}\ar[d]^{Id}\ar[r]^{i}& {E}\ar[d]^{\varphi}\ar[r]^{\pi} & {V}\ar[d]^{Id} \\
{A}\ar[r]^{i}&{E}\ar[r]^{\pi} &{V} },$$
where $\pi: E\rightarrow V$ is the natural projection of $E=A\oplus V$ onto $V$ and
$i: A\rightarrow E$ is the inclusion map. We say that
$\varphi: E\rightarrow E$ \emph{stabilizes} $A$ (resp. \emph{co-stabilizes} $V$) if the left square
(resp. the right square) of the  above diagram is commutative.

Let $\circ$ and $\circ^{'}$ be two left-symmetric (resp. Novikov) structures on $E$ both containing $A$ as a left-symmetric (resp. Novikov) subalgebra.
If there exists a left-symmetric (resp. Novikov) algebra isomorphism $\varphi: (E, \circ)\rightarrow (E,\circ^{'})$ which stabilizes
$A$,  $\circ$ and $\circ^{'}$ are called \emph{equivalent}, which is denoted by
$(E,\circ)\equiv (E,\circ^{'})$.

If there exists a left-symmetric (resp. Novikov) algebra isomorphism $\varphi: (E, \circ)\rightarrow (E,\circ^{'})$ which stabilizes
$A$ and co-stabilizes $V$,   $\circ$ and $\circ^{'}$ are called \emph{cohomologous}, which is denoted by
$(E,\circ)\approx (E,\circ^{'})$.
\end{definition}

Obviously, $\equiv $ and $\approx $ are equivalence relations on the set of all left-symmetric (resp. Novikov) algebra structures on
$E$ containing $A$ as a left-symmetric (resp. Novikov) subalgebra. Denote by $\text{Extd}(E,A)$ (resp. $\text{Extd}^{'}(E,A)$) the set
of all equivalence classes via $\equiv $ (resp. $\approx $). Thus, $\text{Extd}(E,A)$ is the classifying object of the extending structures problem and $\text{Extd}^{'}(E,A)$ provides a classification of the extending structures problem from the point of the view of the extension problem. In addition, it is easy to see that there exists a canonical projection $\text{Extd}(E,A)\twoheadrightarrow \text{Extd}^{'}(E,A)$.

\section{Unified products for left-symmetric algebras}
In this section, we will introduce the definition of unified product for left-symmetric (resp. Novikov) algebras and give a theoretical answer to the
extending structures problem.
\begin{definition}
Let $(A,\circ)$ be a left-symmetric (resp. Novikov) algebra  and $V$ a vector space. An \emph{extending datum} of $A$ by $V$ is
a system $\Omega(A, V)=(l_A, r_A, l_V, r_V, f, \cdot)$ consisting four linear maps and two bilinear maps as follows:
\begin{eqnarray*}
l_A,~~ r_A: A\rightarrow gl(V),~~~l_V,~~ r_V: V\rightarrow gl(A),~~~f: V\times V\rightarrow A,~~~\cdot: V\times V\rightarrow V.
\end{eqnarray*}
Let $\Omega(A, V)=(l_A, r_A, l_V, r_V, f, \cdot)$ be an extending datum. Denote by $A\natural_{\Omega(A, V)}V=A\natural V$ the vector space
$A\times V$ with the bilinear map $\circ : (A\times V)\times (A\times V)\rightarrow A\times V$ defined by
\begin{gather}\label{e3}
(a,x)\circ (b,y)\nonumber\\
=(a\circ b+l_V(x)b+r_V(y)a+f(x,y), x\cdot y+l_A(a)y+r_A(b)x),
\end{gather}
for all $a$, $b\in A$, $x$, $y\in V$.
$A\natural V$ is called the \emph{unified product} of $A$ and $\Omega(A, V)$ if it is a left-symmetric (resp. Novikov) algebra with the product given by (\ref{e3}). In this case, the extending datum $\Omega(A, V)$ is called a \emph{left-symmetric (resp. Novikov) extending structure} of $A$ by $V$.
We denote by $\mathfrak{T}(A,V)$ the set of all left-symmetric (resp. Novikov) extending structures of $A$ by $V$.
\end{definition}

By (\ref{e3}), the following relations hold in $A\natural V$ for any $a$, $b\in A$, $x$, $y\in V$:
\begin{eqnarray}
(a,0)\circ (b,0)=(a\circ b,0),~~~(a,0)\circ (0,y)=(r_V(y)a,l_A(a)y),\\
(0,x)\circ (b,0)=(l_V(x)b,r_A(b)x),~~~(0,x)\circ (0,y)=(f(x,y),x\cdot y).
\end{eqnarray}

\begin{theorem}\label{t1}
Let $A$ be a left-symmetric algebra, $V$ be a vector space and $\Omega(A,V)$ an extending datum of $A$ by $V$. Then,
$A\natural V$ is a unified product if and only if the following conditions hold for any $a$, $b\in A$, $x$, $y$, $z\in V$:
\begin{eqnarray*}
(L1)&&l_V(x)(a\circ b)=-l_V(l_A(a)x-r_A(a)x)b+(l_V(x)a-r_V(x)a)\circ b\\
&&+r_V(r_A(b)x)a+a\circ(l_V(x)b),\\
(L2)&&l_A(a)r_A(b)x-r_A(b)l_A(a)x=r_A(a\circ b)x-r_A(b)r_A(a)x,\\
(L3)&&r_V(x)(a\circ b-b\circ a)=r_V(l_A(b)x)a-r_V(l_A(a)x)b+a\circ (r_V(x)b)-b\circ (r_V(x)a),\\
(L4)&&l_A(a\circ b-b\circ a)x=l_A(a)l_A(b)x-l_A(b)l_A(a)x,\\
(L5)&&r_V(x\cdot y)a=r_V(y)(r_V(x)a-l_V(x)a)+l_V(x)r_V(y)a\\
&&+f(l_A(a)x,y)+f(x,l_A(a)y)-a\circ f(x,y)-f(r_A(a)x,y),\\
(L6)&&l_A(a)(x\cdot y)=-l_A(l_V(x)a-r_V(x)a)y+(l_A(a)x-r_A(a)x)\cdot y\\
&&+r_A(r_V(y)a)x+x\cdot(l_A(a)y),\\
(L7)&&l_V(x\cdot y-y\cdot x)a=(l_V(x)l_V(y)-l_V(y)l_V(x))a\\
&&-(f(x,y)-f(y,x))\circ a+f(x,r_A(a)y)-f(y,r_A(a)x),\\
(L8)&&r_A(a)(x\cdot y-y\cdot x)=r_A(l_V(y)a)x-r_A(l_V(x)a)y\\
&&+x\cdot(r_A(a)y)-y\cdot(r_A(a)x),\end{eqnarray*}
\begin{eqnarray*}
(L9)&&f(x\cdot y,z)-f(x,y\cdot z)-f(y\cdot x,z)+f(y,x\cdot z)+r_V(z)(f(x,y)-f(y,x))\\
&&-l_V(x)f(y,z)+l_V(y)f(x,z)=0,\\
(L10)&&(x\cdot y)\cdot z-x\cdot (y\cdot z)-(y\cdot x)\cdot z+y\cdot (x\cdot z)\\
&&+l_A(f(x,y)-f(y,x))z-r_A(f(y,z))x+r_A(f(x,z))y=0.
\end{eqnarray*}
\end{theorem}
\begin{proof}
Define
\begin{eqnarray*}
R((a,x),(b,y),(c,z))=((a,x),(b,y),(c,z))-((b,y),(a,x),(c,z)),
\end{eqnarray*}
where $(a,x)$, $(b,y)$, $(c,z)\in A\times V$.
Note that $A\natural V$ is a left-symmetric algebra if and only if $R((a,x),(b,y),(c,z))=0$
for all $a$, $b$, $c\in A$ and $x$, $y$, $z\in V$.

Obviously, $R((a,x),(b,y),(c,z))=0$ holds if and only if $R((a,0),(b,0),(c,0))=0$, $R((0,x),(b,0),(c,0))=R((a,0),(0,y),(c,0))=0$,
$R((a,0),(b,0),(0,z))=0$, $R((a,0),(0,y),(0,z))=R((0,x),(b,0),(0,z))=0$, $R((0,x),(0,y),(c,0))=0 $ and $R((0,x),(0,y),(0,z))=0$
are satisfied for all $a$, $b$, $c\in A$ and $x$, $y$, $z\in V$.
By a direct computation, it is easy to see that for any $a$, $b$, $c\in A$ and $x$, $y$, $z\in V$,
$R((a,0),(b,0),(c,0))=0$ is equivalent to that $A$ is a left-symmetric algebra;
$R((0,x),(b,0),(c,0))=0$ $\Leftrightarrow$ $R((a,0),(0,y),(c,0))=0$ $\Leftrightarrow$ $R((0,x),(a,0),(b,0))=0$ $\Leftrightarrow$
(L1) and (L2) hold;
$R((a,0),(b,0),(0,z))=0$ $\Leftrightarrow$ $ R((a,0),(b,0),(0,x))=0 $ $\Leftrightarrow$ (L3) and (L4) hold;
$R((a,0),(0,y),(0,z))=0$ $\Leftrightarrow$ $R((0,x),(b,0),(0,z))=0$ $\Leftrightarrow$ $ R((a,0),(0,x),(0,y))=0$ $\Leftrightarrow$
(L5) and (L6) hold;
$R((0,x),(0,y),(c,0))=0 $ $\Leftrightarrow$ $ R((0,x),(0,y),(a,0))=0$ $\Leftrightarrow$ (L7) and (L8) hold;
$R((0,x),(0,y),(0,z))=0 $ $\Leftrightarrow$  (L9) and (L10) hold.

Thus, this theorem is obtained.

\end{proof}
\begin{remark}
In fact, (L2) and (L4) mean that $(l_A,r_A,V)$ is a bimodule of $A$. Moreover, (L1), (L3), (L6) and (L8) are exactly the compatible conditions defining
a matched pair of left-symmetric algebras (see Theorem 3.5 in \cite{Bai}).
\end{remark}
\begin{corollary}
Let $\Omega(A, V)=(l_A, r_A, l_V, r_V, f, \cdot)$ be a left-symmetric extending structure of $A$ by $V$.
Define $\lhd: V\otimes\mathfrak{ g}(A)\rightarrow V$ by $x \lhd a=l_A(a)x-r_A(a)x$,
$\rhd: V\otimes \mathfrak{g}(A)\rightarrow \mathfrak{g}(A)$ by $x\rhd a=l_V(x)a-r_V(x)a$,
$g: V\otimes V\rightarrow \mathfrak{g}(A)$ by $g(x,y)=f(x,y)-f(y,x)$ and
$\{\cdot,\cdot\}: V\otimes V\rightarrow V$ by $\{x,y\}=x\cdot y-y\cdot x$
for all $a\in A$, $x$, $y\in V$. Then,
$\Omega(\mathfrak{g}(A), V)=(\lhd, \rhd, f,\{\cdot,\cdot\})$ is a Lie extending structure of $\mathfrak{g}(A)$ by $V$ (see Definition 3.1 in \cite{AM4}).

\end{corollary}
\begin{proof}
This conclusion can be obtained directly from the relation between left-symmetric algebra $A\natural V$ and its
sub-adjacent Lie algebra. In fact, the sub-adjacent Lie algebra $\mathfrak{g}(A\natural V)$ is just the Lie algebra obtained from
the Lie extending structure $\Omega(\mathfrak{g}(A), V)=(\lhd, \rhd, g,\{\cdot,\cdot\})$:
\begin{eqnarray*}
[(a,x),(b,y)]=([a,b]+x\rhd b-y\rhd a+g(x,y), \{x,y\}+x\lhd a-y\lhd b)
\end{eqnarray*}
for all $a$, $b\in A$, $x$, $y\in V$.

\end{proof}
\begin{corollary}\label{co1}
Let $A$ be a Novikov algebra, $V$ be a vector space and $\Omega(A,V)$ an extending datum of $A$ by $V$. Then,
$A\natural V$ is a unified product of Novikov algebras if and only if (L1)-(L10) and the following conditions hold for any $a$, $b\in A$, $x$, $y$, $z\in V$:
\begin{eqnarray*}
(N1)~~&&(l_V(x)a)\circ b+l_V(r_A(a)x)b=(l_V(x)b)\circ a+l_V(r_A(b)x)a,\\
(N2)~~&&r_A(b)r_A(a)x=r_A(a)r_A(b)x,\\
(N3)~~&&(r_V(x)a)\circ b+l_V(l_A(a)x)b=r_V(x)(a\circ b),\\
(N4)~~&&r_A(b)l_A(a)x=l_A(a\circ b)x,\\
(N5)~~&&r_V(y)r_V(x) a+f(l_A(a)x,y)=r_V(x)r_V(y) a+f(l_A(a)y,x),\\
(N6)~~&&l_A(r_V(x)a)y+(l_A(a)x)\cdot y=l_A(r_V(y)a)x+(l_A(a)y)\cdot x,\\
(N7)~~&&r_V(y)(l_V(x)a)+f(r_A(a)x,y)=f(x,y)\circ a+l_V(x\cdot y)a,\\
(N8)~~&&l_A(l_V(x)a)y+(r_A(a)x)\cdot y=r_A(a)(x\cdot y),\\
(N9)~~&&r_V(z)f(x,y)+f(x\cdot y,z)=r_V(y)f(x,z)+f(x\cdot z,y),\\
(N10)~~&&l_A(f(x,y))z+(x\cdot y)\cdot z=l_A(f(x,z))y+(x\cdot z)\cdot y.
\end{eqnarray*}
\end{corollary}
\begin{proof}
According to Theorem \ref{t1}, $A\natural V$ is a Novikov algebra if and only if (L1)-(L10) hold and
\begin{eqnarray}\label{eq1}
((a,x)\cdot (b,y))\cdot (c,z)=((a,x)\cdot (c,z))\cdot (b,y),
\end{eqnarray}for all $a$, $b$, $c\in A$ and $x$, $y$, $z\in V$.
Then, similar to that in the proof of Theorem \ref{t1}, it is easy to check that
(\ref{eq1}) holds if and only if (N1)-(N10) are satisfied. Thus, we obtain this corollary.

\end{proof}
\begin{remark}
In fact, (L2), (L4), (N2) and (N4) means that $(l_A,r_A,V)$ is a bimodule of the Novikov algebra $A$.
\end{remark}
\begin{example}
Let $\Omega(A, V)=(l_A, r_A, l_V, r_V, f, \cdot)$ be an extending datum of a left-symmetric algebra $A$ by a vector space $V$ where
$l_A$, $r_A$, $l_V$, $r_V$ are trivial maps. Then, $\Omega(A, V)=(f, \cdot)$ is a left-symmetric extending structure of $A$ by $V$ if and only if $(V,\cdot)$ is a left-symmetric algebra and $f: V\times V\rightarrow A$ satisfies
\begin{eqnarray}\label{eqnn1}
&&a\circ f(x,y)=(f(x,y)-f(y,x))\circ a=0,\\
\label{eqnn2}&&f(x\cdot y,z)-f(x,y\cdot z)-f(y\cdot x,z)+f(y,x\cdot z)=0,
\end{eqnarray}
for all $a\in A$, $x$, $y$, $z\in V$. Moreover, if $A$ is a Novikov algebra, $\Omega(A, V)=(f, \cdot)$ is a Novikov extending structure of $A$ by $V$ if and only if $(V,\cdot)$ is a Novikov algebra and $f: V\times V\rightarrow A$ satisfies  (\ref{eqnn2}) and
\begin{eqnarray*}
&&a\circ f(x,y)=f(x,y)\circ a=0,\\
&&f(x\cdot y,z)=f(x\cdot z,y), ~~~~x,~y,~z\in V.
\end{eqnarray*}
The associated unified product $A\natural V$ denoted by $A\natural_f V$ is called the \emph{twisted product} of $A$ and $V$.
The product on $A\natural_f V$ is given by for any $a$, $b\in A$, $x$, $y\in V$:
$$(a,x)\circ (b,y)=(a\circ b+f(x,y),x\cdot y).$$
\end{example}

\begin{example}\label{ex1}
Let $\Omega(A, V)=(l_A, r_A, l_V, r_V, f, \cdot)$ be an extending datum of a left-symmetric algebra $A$ by a vector space $V$ where
$l_A$ and $r_A$ are trivial maps. Then, $\Omega(A, V)=(l_V, r_V, f, \cdot)$ is a left-symmetric extending structure of $A$ by $V$ if and only if $(V,\cdot)$ is a left-symmetric algebra and the following conditions are satisfied:
\begin{eqnarray*}
(C1)&&l_V(x)(a\circ b)=(l_V(x)a-r_V(x)a)\circ b+a\circ(l_V(x)b),\\
(C2)&&r_V(x)(a\circ b-b\circ a)=a\circ (r_V(x)b)-b\circ (r_V(x)a),\\
(C3)&&r_V(x\cdot y)a=r_V(y)(r_V(x)a-l_V(x)a)+l_V(x)r_V(y)a-a\circ f(x,y),\\
(C4)&&l_V(x\cdot y-y\cdot x)a=(l_V(x)l_V(y)-l_V(y)l_V(x))a-(f(x,y)-f(y,x))\circ a,\\
(C5)&&f(x\cdot y,z)-f(x,y\cdot z)-f(y\cdot x,z)+f(y,x\cdot z)+r_V(z)(f(x,y)-f(y,x))\\
&&-l_V(x)f(y,z)+l_V(y)f(x,z)=0,
\end{eqnarray*}
for all $a$, $b\in A$ and $x$, $y$, $z\in V$. Moreover, if $A$ is a Novikov algebra, $\Omega(A, V)=(l_V, r_V, f, \cdot)$ is a Novikov extending structure of $A$ by $V$ if and only if $(V,\cdot)$ is a Novikov algebra and $f: V\times V\rightarrow A$ satisfies (C1)-(C5) and
\begin{eqnarray*}
(CN1)~~&&(l_V(x)a)\circ b=(l_V(x)b)\circ a,\\
(CN2)~~&&(r_V(x)a)\circ b=r_V(x)(a\circ b),\\
(CN3)~~&&r_V(y)r_V(x) a=r_V(x)r_V(y) a,\\
(CN4)~~&&r_V(y)(l_V(x)a)=f(x,y)\circ a+l_V(x\cdot y)a,\\
(CN5)~~&&r_V(z)f(x,y)+f(x\cdot y,z)=r_V(y)f(x,z)+f(x\cdot z,y),
\end{eqnarray*}
for all $a$, $b\in A$ and $x$, $y$, $z\in V$. The associated unified product $A\natural V$ denoted by $A\natural_{l_V,r_V}^f V$ is called the \emph{crossed product} of $A$ and $V$.
The product on $A\natural_{l_V,r_V}^f V$ is given by for any $a$, $b\in A$, $x$, $y\in V$:
$$(a,x)\circ (b,y)=(a\circ b+l_V(x)b+r_V(y)a+f(x,y),x\cdot y).$$
Obviously, $A$ is an ideal of $A\natural_{l_V,r_V}^f V$.
\end{example}
\begin{example}\label{ex2}
Let $\Omega(A, V)=(l_A, r_A, l_V, r_V, f, \cdot)$ be an extending datum of a left-symmetric algebra $A$ by a vector space $V$ where
$f$ is a trivial map. Then, $\Omega(A, V)=(l_A, r_A, l_V, r_V, f)$ is a left-symmetric extending structure of $A$ by $V$ if and only if
$(V,\cdot)$ is a left-symmetric algebra, $(l_A, r_A, V)$ is a bimodule of $A$, $(l_V,r_V, A)$ is a bimodule of $V$ and they satisfy (L1), (L3),
(L6) and (L8). Moreover, if $A$ is a Novikov algebra, $\Omega(A, V)=(l_V, r_V, f, \cdot)$ is a Novikov extending structure of $A$ by $V$ if and only if $(V,\cdot)$ is a Novikov algebra, $(l_A, r_A, V)$ is a bimodule of the Novikov algebra $A$, $(l_V,r_V, A)$ is a bimodule of the Novikov algebra  $V$, and
they satisfy (L1), (L3), (L6), (L8), (N1), (N3), (N6) and (N8). The associated unified product $A\natural V$ denoted by $A\bowtie_{l_V,r_V}^{l_A,r_A} V$ is called the \emph{bicrossed product} of $A$ and $V$ (see \cite{Bai}). The product on  $A\bowtie_{l_V,r_V}^{l_A,r_A} V$ is given by for any $a$, $b\in A$, $x$, $y\in V$:
$$(a,x)\circ (b,y)=(a\circ b+l_V(x)b+r_V(y)a,x\cdot y+l_A(a)y+r_A(b)x).$$
Here, both $A$ and $V$ are subalgebras of $A\bowtie_{l_V,r_V}^{l_A,r_A} V$.
\end{example}

Given a left-symmetric (resp. Novikov) extending structure $\Omega(A,V)=(l_A,r_A,$ $l_V,r_V,f,\cdot)$. It is obvious that
$A$ can be seen a left-symmetric (resp.Novikov) subalgebra of $A\natural V$. In fact, we can also prove that any left-symmetric (resp. Novikov)
algebra structure on a vector space $E$ containing $A$ as a subalgebra is isomorphic to a unified product.

\begin{theorem}\label{pt1}
Let $A$ be a left-symmetric (resp. Novikov) algebra, $E$ a vector space containing $A$ as a subspace and $\circ$
a left-symmetric (resp. Novikov)  algebra structure on $E$ such that
$A$ is a subalgebra of $(E,\circ)$. Then, there exists a left-symmetric (resp. Novikov)  extending structure $\Omega(A,V)=(l_A,r_A,l_V,r_V,f,\cdot)$ of $A$ by a subspace $V$ of $E$ and an isomorphism of left-symmetric (resp. Novikov) algebras
$(E,\circ)\cong A\natural V$ which stabilizes $A$ and co-stabilizes $V$.
\end{theorem}

\begin{proof}
Note that there is a natural linear map $p: E\rightarrow A$ such that $p(a)=a$ for all $a\in A$.
Set $V=\text{Ker}(p)$ which is a complement of $A$ in $E$. Then, we present a extending datum $\Omega(A,V)=(l_A,r_A,l_V,r_V,f,\cdot)$ of $A$ by a subspace $V$ of $E$ defined as follows:
\begin{eqnarray*}
&&l_A: A\rightarrow gl(V), ~~l_A(a)v:=a\circ v-p(a\circ v),\\
&& r_A: A\rightarrow gl(V),~~r_A(a)v:=v\circ a-p(v\circ a),\\
&&l_V: V\rightarrow gl(A),~~l_V(v)a:=p(v\circ a),~~~~r_V: V\rightarrow gl(A),~~r_V(v)a:=p(a\circ v),\\
&&f:V\times V\rightarrow A,~~f(x,y):=p(x\circ y),~~\cdot: V\times V\rightarrow V,~~x\cdot y:=x\circ y-p(x\circ y),
\end{eqnarray*}
for any $a\in A$, $x$, $y\in V$. It is easy to see that $\varphi: A\times V\rightarrow E$ defined as
$\varphi(a,x)=a+x$ is a linear isomorphism, whose inverse is as follows: $\varphi^{-1}(e):=(p(e),e-p(e))$ for all $e\in E$.
Next, we should prove that $\Omega(A,V)=(l_A,r_A,l_V,r_V,f,\cdot)$ is a left-symmetric (resp. Novikov) extending structure of $A$ by $V$ and $\varphi: A \natural V \rightarrow E$ is an isomorphism of left-symmetric (resp. Novikov) algebras that stabilizes $A$ and co-stabilizes $V$. In fact, if $\varphi: A\times V\rightarrow E$ is an isomorphism of left-symmetric (resp. Novikov) algebras, there exists
a unique left-symmetric (resp. Novikov) product given by
\begin{eqnarray}\label{ee1}
(a,x)\circ (b,y)=\varphi^{-1}(\varphi(a,x)\circ \varphi(b,y))
\end{eqnarray}
for all $a$, $b\in A$ and $x$, $y\in V$. Therefore, for completing the proof, we only need to check that the product defined by
(\ref{ee1}) is just the one given by (\ref{e3}) related with the above extending system $\Omega(A,V)=(l_A,r_A,l_V,r_V,f,\cdot)$.
In detail, we get
\begin{eqnarray*}
(a,x)\circ (b,y)&=&\varphi^{-1}(\varphi(a,x)\circ \varphi(b,y))=\varphi^{-1}((a+x)\circ (b+y))\\
&=&\varphi^{-1}(a\circ b+a\circ y+x\circ b+x\circ y)\\
&=&(a\circ b+p(a\circ y)+p(x\circ b)+p(x\circ y),\\
&&a\circ y+x\circ b+x\circ y-p(a\circ y)-p(x\circ b)-p(x\circ y))\\
&=&(a\circ b+l_v(x)b+r_v(y)a+f(x,y), x\cdot y+l_A(a)y+r_A(b)x),
\end{eqnarray*}
for all $a$, $b\in A$ and $x$, $y\in V$.
Therefore, $\varphi: A\natural V\rightarrow E$ is an isomorphism of left-symmetric (resp. Novikov) algebras and
the following diagram is commutative
$$\xymatrix{ {A}\ar[d]^{Id}\ar[r]^{i}& {A\natural V}\ar[d]^{\varphi}\ar[r]^{\rho} & {V}\ar[d]^{Id} \\
{A}\ar[r]^{i}&{E}\ar[r]^{\pi} &{V} },$$
where $\rho:A\natural V\rightarrow V$ and $\pi: E\rightarrow V$ are the natural projections.
Then, the proof is finished.
\end{proof}
\begin{remark}\label{rr1}
By Example \ref{ex1} and Theorem \ref{pt1}, it is easy to obtain the following result:\\
\emph{Let $A$ be a left-symmetric (resp. Novikov) algebra, $E$ a vector space containing $A$ as a subspace. Then, any left-symmetric (resp. Novikov) algebra
structure on $E$ containing $A$ as an ideal is isomorphic to a crossed product of left-symmetric (resp. Novikov) algebras $A\natural_{l_V,r_V}^f V$.}
\\

Similarly, by Example \ref{ex2} and Theorem \ref{pt1}, we can get (This result can also be referred to Theorem 3.5 in \cite{Bai}):\\
\emph{Let $A$ and $B$ be two left-symmetric (resp. Novikov) algebras. Then, any left-symmetric (resp. Novikov) algebra
structure on $E$ which is a direct sum of the underlying vector of two subalgebras $A$ and $B$ is isomorphic to a bicrossed product of left-symmetric (resp. Novikov) algebras $A\bowtie_{l_V,r_V}^{l_A,r_A} B$.}

\end{remark}

Next, by Theorem \ref{pt1}, for classifying all left-symmetric (resp. Novikov) algebra structures on $E$ containing $A$ as a subalgebra,
we only need to classify all unified products $A\natural V$ associated to all left-symmetric (resp. Novikov) algebra structures
$\Omega(A,V)=(l_A,r_A,l_V,r_V,f,\cdot)$ for a given complement $V$ of $A$ in $E$.

\begin{lemma}\label{pt2}
Let $\Omega(A,V)=(l_A,r_A,l_V,r_V,f,\cdot)$ and $\Omega^{'}(A,V)=(l_A^{'},r_A^{'},l_V^{'},r_V^{'},f^{'},\cdot^{'})$ be two
left-symmetric (resp. Novikov) extending structures of $A$ by $V$ and $A\natural V$, $A\natural^{'} V$ be the corresponding
unified products. Then, there is a bijection between the set of all morphisms of left-symmetric (resp. Novikov) algebras
$\varphi: A\natural V\rightarrow A\natural^{'} V$ which stabilizes $A$ and the set of pairs $(\lambda,\mu)$, where
$\lambda: V\rightarrow A$, $\mu: V\rightarrow V$ are linear maps satisfying the following conditions
\begin{gather}
\label{er1}r_V(x)a+\lambda(l_A(a)x)=a\circ \lambda(x)+r_V^{'}(\mu(x))a,\\
\label{er2}\mu(l_A(a)x)=l_A^{'}(a)\mu(x),\\
\label{er3}l_V(x)a+\lambda(r_A(a)x)=\lambda(x)\circ a+l_V^{'}(\mu(x))a,\\
\label{er4}\mu(r_A(a)x)=r_A^{'}(a)\mu(x),\\
f(x,y)+\lambda(x\cdot y)=\lambda(x)\circ \lambda(y)+l_V^{'}(\mu(x))\lambda(y)\nonumber\\
\label{er5}+r_V^{'}(\mu(y))\lambda(x)+f^{'}(\mu(x),\mu(y)),\\
\label{er6}\mu(x\cdot y)=\mu(x)\cdot^{'}\mu(y)+l_A^{'}(\lambda(x))\mu(y)+r_A^{'}(\lambda(y))\mu(x),
\end{gather}
for any $a\in A$, $x$, $y\in V$.\\

The bijection from the set of all morphisms of left-symmetric (resp. Novikov) algebras
$\varphi_{\lambda,\mu}: A\natural V\rightarrow A\natural^{'} V$ to the set of pairs $(\lambda,\mu)$ is given as follows:
\begin{eqnarray*}
\varphi(a,x)=(a+\lambda(x),\mu(x)).
\end{eqnarray*}
In addition, $\varphi_{\lambda,\mu}$ is an isomorphism if and only if $\mu: V\rightarrow V$ is an isomorphism, and
$\varphi_{\lambda,\mu}$ costabilizes $V$ if and only if $\mu= Id_V$.
\end{lemma}
\begin{proof}
Let $\varphi: A\natural V\rightarrow A\natural^{'} V$ be a morphism of left-symmetric (resp. Novikov) algebras.
Since $\varphi$ stabilizes $A$, $\varphi(a,0)=(a,0)$. Moreover, we can set $\varphi(0,x)=(\lambda(x),\mu(x))$, where
$\lambda: V\rightarrow A$, $\mu: V\rightarrow V$ are two linear maps. Therefore,
we get $\varphi(a,x)=(a+\lambda(x),\mu(x))$. Then, we should prove that
$\varphi$ is a morphism of left-symmetric (resp. Novikov) algebras if and only if (\ref{er1})-(\ref{er6}) hold.
It is enough to check that
\begin{eqnarray}
\label {er7}\varphi((a,x)\circ (b,y))=\varphi(a,x)\circ \varphi(b,y)
\end{eqnarray} holds for all generators of $A\natural V$.
Obviously, (\ref{er7}) holds for the pair $(a,0)$, $(b,0)$.
Then, we consider (\ref{er7}) for the pair $(a,0)$, $(0,x)$.
According to \begin{eqnarray*}
\varphi((a,0)\circ (0,x))&=&\varphi(r_V(x)a,l_A(a)x)\\
&=&(r_V(x)a+\lambda(l_A(a)x),\mu(l_A(a)x),
\end{eqnarray*}
and
\begin{eqnarray*}
\varphi(a,0)\circ \varphi(0,x)&=&(a,0)\circ (\lambda(x),\mu(x))\\
&=&(a\circ \lambda(x)+r_V^{'}(\mu(x))a,l_A^{'}(a)\mu(x)),
\end{eqnarray*}
we get that (\ref{er7}) holds for the pair $(a,0)$, $(0,x)$ if and only if (\ref{er1}) and (\ref{er2}) hold.
Similarly, it is easy to check that (\ref{er7}) holds for the pair $(0,x)$, $(a,0)$ if and only if (\ref{er3}) and (\ref{er4}) hold;
(\ref{er7}) holds for the pair $(0,x)$, $(0,y)$ if and only if (\ref{er5}) and (\ref{er6}) hold.

Assume that $\varphi_{\lambda,\mu}$ is bijective. It is obvious that $\mu$ is surjective. Then, we only need to prove that
$\mu$ is injective. Let $x\in V$ such that $\mu(x)=0$. Then, we get
$\varphi_{\lambda,\mu}(-\lambda(x),x)=(-\lambda(x)+\lambda(x),\mu(x))=(0,0)$. Thus, $x=0$, i.e. $\mu$ is injective.
Conversely, assume that $\mu:V\rightarrow V$ is bijective. Then, $\varphi_{\lambda,\mu}$ has the inverse given by
$\varphi_{\lambda,\mu}^{-1}(a,x)=(a-\lambda(\mu^{-1}(x)),\mu^{-1}(x))$. Thus, by the first part, $\varphi_{\lambda,\mu}$
is an isomorphism. Therefore, $\varphi_{\lambda,\mu}$ is an isomorphism if and only if $\mu: V\rightarrow V$ is an isomorphism.
Finally, it is obvious that $\varphi_{\lambda,\mu}$ costabilizes $V$ if and only if $\mu= Id_V$.

Now, the proof is finished.

\end{proof}

\begin{definition}
If there exists a pair of linear maps $(\lambda,\mu)$ where $\lambda: V\rightarrow A$, $\mu\in Aut_k(V)$ such that
the left-symmetric (resp. Novikov) extending structure $\Omega(A,V)=(l_A,r_A,l_V,r_V,f,\cdot)$ can be obtained from
another corresponding extending structure  $\Omega^{'}(A,V)=(l_A^{'},r_A^{'},l_V^{'},r_V^{'},f^{'},\cdot^{'})$ using
$(\lambda,\mu)$ as follows:
\begin{eqnarray}
&&r_A(a)x=\mu^{-1}(r_A^{'}(a)\mu(x)),\\
&&l_A(a)x=\mu^{-1}(l_A^{'}(a)\mu(x)),\\
&&l_V(x)a=\lambda(x)\circ a+l_V^{'}(\mu(x))a-\lambda(r_A(a)x),\\
&&r_V(x)a=a\circ \lambda(x)+r_V^{'}(\mu(x))a-\lambda(l_A(a)x),\\
&&x\cdot y=\mu^{-1}(\mu(x)\cdot^{'}\mu(y))+\mu^{-1}(l_A^{'}(\lambda(x))\mu(y))+\mu^{-1}(r_A^{'}(\lambda(y))\mu(x)),\\
&&f(x,y)=\lambda(x)\circ \lambda(y)+l_V^{'}(\mu(x))\lambda(y)\\
&&+r_V^{'}(\mu(y))\lambda(x)+f^{'}(\mu(x),\mu(y))-\lambda(x\cdot y),\nonumber
\end{eqnarray}
for any $a\in A$, $x$, $y\in V$, then $\Omega(A,V)$ and  $\Omega^{'}(A,V)$ are called \emph{equivalent} and we denote it
by $\Omega(A,V)\equiv \Omega^{'}(A,V)$.

Moreover, in special, $\mu=Id$, $\Omega(A,V)$ and  $\Omega^{'}(A,V)$ are called \emph{cohomologous} and we denote it
by $\Omega(A,V)\approx \Omega^{'}(A,V)$.

\end{definition}

Then, by the above discussion, the answer for the extending structures problem of left-symmetric algebras (resp. Novikov algebras) is given as follows:
\begin{theorem}\label{th1}
Let $A$ be a left-symmetric algebra (resp. Novikov algebra), $E$ a vector space that contains $A$ as a subspace and $V$ a complement
of $A$ in $E$. Then, we get:\\
(1)Denote $\mathcal{H}_A^2(V,A):=\mathfrak{T}(A,V)/\equiv$. Then, the map
\begin{eqnarray}
\mathcal{H}_A^2(V,A)\rightarrow Extd(E,A),~~~~\overline{\Omega(A,V)}\rightarrow (A\natural V,\circ)
\end{eqnarray}
is bijective, where $\overline{\Omega(A,V)}$ is the equivalence class of $\Omega(A,V)$ under $\equiv$.\\
(2) Denote $\mathcal{H}^2(V,A):=\mathfrak{T}(A,V)/\approx$. Then, the map
\begin{eqnarray}
\mathcal{H}^2(V,A)\rightarrow Extd^{'}(E,A),~~~~\overline{\overline{\Omega(A,V)}}\rightarrow (A\natural V,\circ)
\end{eqnarray}
is bijective, where $\overline{\overline{\Omega(A,V)}}$ is the equivalence class of $\Omega(A,V)$ under $\approx$.\\
\end{theorem}
\begin{proof}
This theorem can be directly obtained from Theorem \ref{t1}, Corollary \ref{co1}, Theorem \ref{pt1} and Lemma \ref{pt2}.
\end{proof}
\section{Flag extending structures of left-symmetric algebras}
In this section, we will study a class of special extending structures of left-symmetric (resp. Novikov) algebras in detail.
\begin{definition}\label{def1}
Let $A$ be a left-symmetric (resp. Novikov) algebra and $E$ a vector space containing $A$ as a subspace.
If there exists a finite chain of left-symmetric (resp. Novikov) subalgebras of $E$:
$A=E_0\subset E_1\subset \cdots \subset E_m=E$ such that $E_i$ has codimension 1 in $E_{i+1}$ for all
$i=0$, $\cdots$, $m-1$, then the left-symmetric (resp. Novikov) algebra structure on
$E$ is called a \emph{flag extending structure } of $A$.
\end{definition}

It is easy to see that $\text{dim}_k(V)=m$ in Definition \ref{def1}, where $V$ is the complement of $A$ in $E$. In fact,
all flag extending structures of $A$ in $E$ can be completely described by a recursive process. First,
classify all unified products of $A\natural V_1$ where $V_1$ is a 1-dimensional vector space.
Second, replacing $A$ by $A\natural V_1$, characterize all unified products of $(A\natural V_1)\natural V_2$ where $V_2$ is a 1-dimensional vector space. Then, we can iterate this process. After $m$ steps, we can describe and classify all flag extending structures
of $A$ in $E$. Therefore, in this section, we mainly study the extending structure of $A$ by a 1-dimensional vector space $V$.

\begin{definition}\label{dd1}
Let $A$ be a left-symmetric algebra. A \emph{flag datum} of $A$ is a 6-tuple $(h,g,D,T,a_0, \alpha)$ where $a_0\in A$,
$\alpha\in k$, $g: A\rightarrow k$ is a morphism of algebras, $h: A\rightarrow k$, $D$ and $T: A\rightarrow A$ are linear maps satisfying
for $a$, $b\in A$:
\begin{eqnarray}
&&\label{es1}h(a\circ b)=h(b\circ a),\\
&&\label{es2}D(a\circ b)=D(a)\circ b+a\circ D(b)+(g(a)-h(a))D(b)+g(b)T(a)-T(a)\circ b,\\
&&\label{es3}T(a\circ b)-T(b\circ a)=h(b)T(a)-h(a)T(b)+a\circ T(b)-b\circ T(a),\\
&&\label{es4}T^2(a)=T(D(a))-D(T(a))+a\circ a_0+(g(a)-2h(a))a_0+\alpha T(a),\\
&&\label{es5}h(D(a))-h(T(a))=g(T(a))+\alpha(h(a)-g(a)).
\end{eqnarray}
Denote by $\mathcal{FL}(A)$ the set of all flag datums of $A$.
\end{definition}

\begin{definition}
Let $A$ be a Novikov algebra. A \emph{flag datum} of $A$ is a 6-tuple $(h,g,D,T,a_0, \alpha)$ where $a_0\in A$,
$\alpha\in k$, $g: A\rightarrow k$ is a morphism of algebras, $h: A\rightarrow k$, $D$ and $T: A\rightarrow A$ are linear maps satisfying
(\ref{es1})-(\ref{es5}) and the following equalities( $a$, $b\in A$):
\begin{eqnarray}
&&\label{df1}D(a)\circ b+g(a)D(b)=D(b)\circ a+g(b)D(a),\\
&&\label{df2}T(a\circ b)=T(a)\circ b+h(a)D(b),\\
&&\label{df3}h(a\circ b)=h(a)g(b),\\
&&\label{df4}T(D(a))=a_0\circ a+\alpha D(a)-g(a)a_0,\\
&&\label{df5}h(D(a))=0.
\end{eqnarray}
Denote by $\mathcal{FN}(A)$ the set of all flag datums of the Novikov algebra $A$.
\end{definition}

\begin{proposition}\label{4p}
Let $A$ be a left-symmetric (resp. Novikov) algebra  and $V$ a vector space of dimension 1 with a basis $\{x\}$. Then, there exists a bijection between the set $\mathfrak{T}(A,V)$ of all left-symmetric (resp. Novikov) extending structures of $A$ by $V$ and $\mathcal{FL}(A)$ (resp. $\mathcal{FN}(A)$).

Through the above bijection, the left-symmetric (resp. Novikov) extending structure $\Omega(A,V)=(l_A,r_A,l_V,r_V,f,\cdot)$ corresponding
to $(h,g,D,T,a_0, \alpha)\in \mathcal{FL}(A)$ (resp. $\in \mathcal{FN}(A)$) is given as follows:
\begin{eqnarray}
&&l_A(a)x=h(a)x,~~~r_A(a)x=g(a)x,~~l_V(x)a=D(a),\\
&&r_V(x)a=T(a),~~~x\cdot x=\alpha x,~~~f(x,x)=a_0,
\end{eqnarray}
for all $a\in A$.
\end{proposition}
\begin{proof}
Given a left-symmetric (resp. Novikov) extending structure $\Omega(A,V)=(l_A,r_A,l_V,r_V,f,\cdot)$.
By the fact that ${\text{dim}}_k(V)=1$, we can set
\begin{eqnarray*}
&&l_A(a)x=h(a)x,~~~r_A(a)x=g(a)x,~~l_V(x)a=D(a),\\
&&r_V(x)a=T(a),~~~x\cdot x=\alpha x,~~~f(x,x)=a_0,
\end{eqnarray*}
where $a_0\in A$, $\alpha\in k$, $g: A\rightarrow k$, $h: A\rightarrow k$, $D$ and $T: A\rightarrow A$ are linear maps.

Then, by a straightforward computation, we can obtain that the conditions $(L1)$-$(L10)$ in Theorem \ref{t1} are equivalent
to the fact that $g: A\rightarrow k$ is a morphism of algebras and $(\ref{es1})-(\ref{es5})$ hold. Moreover,
the conditions $(N1)-(N10)$ are equivalent to the fact that $(\ref{df1})-(\ref{df5})$ hold.
\end{proof}

Let $(h,g,D,T,a_0, \alpha)\in \mathcal{FL}(A)$ (resp. $\in \mathcal{FN}(A)$). By Proposition \ref{4p},
we denote the unified product associated with the extending structure corresponding to $(h,g,D,T,a_0, \alpha)$
by $LS(A,x| (h,g,D,T,a_0, \alpha))$ (resp. $NV(A,x| (h,g,D,T,a_0, \alpha))$).

Next, we present a classification of all left-symmetric (resp. Novikov) algebras on $E$ containing $A$ as a subalgebra whose codimension is 1.

\begin{theorem}\label{tt1}
Let $A$ be a left-symmetric (resp. Novikov) algebra of codimension 1 in the vector space $E$.
Then, we get:\\
(1) $Extd(E,A)\cong \mathcal{H}^2_A(k,A)\cong \mathcal{FL}(A)/\equiv (\text{resp}.\mathcal{FN}(A)/\equiv) $, where $\equiv$ is the equivalence relation on the set
$\mathcal{FL}(A)$ (resp. $\mathcal{FN}(A)$) as follows: $(h,g,D,T,a_0, \alpha)$ $\equiv$ $(h^{'},g^{'},D^{'},T^{'},a^{'}_0, \alpha^{'})$ if and only if
$h=h^{'}$, $g=g^{'}$ and there exists a pair $(\beta,b_0)\in k^{*}\times A$ such that for any $a\in A$:
\begin{eqnarray}
\label{q1} D(a)=b_0\circ a+\beta D^{'}(a)-g(a)b_0,\\
\label{q2}T(a)=a\circ b_0+\beta T^{'}(a)-h(a)b_0,\\
\label{q3}a_0=b_0\circ b_0+\beta D^{'}(b_0)+\beta T^{'}(b_0)+\beta^2 a^{'}_0-\alpha b_0,\\
\label{q4}\alpha=\beta \alpha^{'}+h^{'}(b_0)+g^{'}(b_0).
\end{eqnarray}
The bijection between $\mathcal{FL}(A)/\equiv$ (resp. $\mathcal{FN}(A)/\equiv$ ) and $Extd(E,A)$ is given by
$$\overline{(h,g,D,T,a_0, \alpha)}\mapsto LS(A,x| (h,g,D,T,a_0, \alpha))$$
$$(\text{resp.} \overline{(h,g,D,T,a_0, \alpha)}\mapsto NV(A,x| (h,g,D,T,a_0, \alpha)),$$
where $\overline{(h,g,D,T,a_0, \alpha)}$ is the equivalence class of $\mathcal{FL}(A)$ (resp. $\mathcal{FN}(A)$) under $\equiv$.\\
(2) $Extd^{'}(E,A)\cong \mathcal{H}^2(k,A)\cong \mathcal{FL}(A)/\approx (\text{resp}.\mathcal{FN}(A)/\approx)$, where $\approx$ is the equivalence relation on the set
$\mathcal{FL}(A)$ (resp. $\mathcal{FN}(A)$ ) as follows: $(h,g,D,T,a_0, \alpha)$ $\approx$ $(h^{'},g^{'},D^{'},T^{'},a^{'}_0, \alpha^{'})$ if and only if
$h=h^{'}$, $g=g^{'}$ and there exists a $b_0 \in  A$ such that (\ref{q1})-(\ref{q4}) holds for $\beta=1$.
The bijection between $\mathcal{FL}(A)/\approx$ (resp. $\mathcal{FN}(A)/\approx$ ) and $Extd^{'}(E,A)$ is given by
$$\overline{\overline{(h,g,D,T,a_0, \alpha)}}\mapsto LS(A,x| (h,g,D,T,a_0, \alpha))$$
$$(\text{resp.}\overline{\overline{(h,g,D,T,a_0, \alpha)}}\mapsto NV(A,x| (h,g,D,T,a_0, \alpha)),)$$
where $\overline{\overline{(h,g,D,T,a_0, \alpha)}}$ is the equivalence class of $\mathcal{FL}(A)$ (resp. $\mathcal{FN}(A)$) under $\approx$.
\end{theorem}
\begin{proof}
This theorem can be directly obtained from Lemma \ref{pt2} and Theorem \ref{th1} by some simple computations.
\end{proof}

Finally, we provide two explicit examples to compute $\mathcal{H}^2(k,A)$.

\begin{example}
Let $k$ be a field of characteristic zero and $A$ be a 4-dimensional left-symmetric algebra with a basis
$\{e_1,e_2,e_3,e_4\}$ and the non-zero products given by
\begin{eqnarray}
\label{qw1}e_1\circ e_2=e_4,~~e_3\circ e_2=e_1,~~e_4\circ e_3=2e_3,\\
\label{qw2}e_2\circ e_1=e_4,~~e_4\circ e_1=e_1,~~e_2\circ e_3=e_4,~~e_4\circ e_2=-e_2.
\end{eqnarray}
This left-symmetric algebra is the unique 4-dimensional complete simple left-symmetric algebra up to isomorphisms over $\mathbb{C}$ (see \cite{Bu}).

By some computations, we can obtain that the flag datum $(h,g,D,T,a_0,\alpha)$ of $A$ is given by as follows:
\begin{eqnarray}
h=g=0,~~a_0=dx_4,~~\alpha=e,\\
D=\left(
    \begin{array}{cccc}
      c & 0 & 0 & 0 \\
      0 & b & 0 & 0 \\
      0 & 0 & c & 0 \\
      0 & 0 &0 & b \\
    \end{array}
  \right),~~~T=\left(
    \begin{array}{cccc}
      c & 0 & 0 & 0 \\
      0 & c & 0 & 0 \\
      0 & 0 & b & 0 \\
      0 & 0 &0 & b \\
    \end{array}
  \right),
\end{eqnarray}
where $b$, $c$, $d$, $e\in k$ and $c^2=ec$, $b^2=eb$.

Therefore, any 5-dimensional left-symmetric algebra that contains $A$ as a left-symmetric subalgebra
is isomorphic to the following left-symmetric algebra denoted by $A_{b,c,d,e}$ with the basis
$\{e_1,e_2,e_3,e_4,x\}$ and the products given by (\ref{qw1}), (\ref{qw2}) and
\begin{eqnarray}
\label{qw3}&&e_1\circ x=ce_1,~~x\circ e_1=ce_1,~~e_2\circ x=ce_2,~~x\circ e_2=be_2,\\
\label{qw4}&&e_3\circ x=be_3,~~x\circ e_3=ce_3,~~e_4\circ x=be_4,~~x\circ e_4=be_4,\\
&&x\circ x=de_4+ex,
\end{eqnarray}
where $c^2=ec$, $b^2=eb$.

By Theorem \ref{tt1}, we can get that two such left-symmetric algebras $A_{b,c,d,e}$ and $A_{b^{'},c^{'},d^{'},e^{'}}$
are isomorphic if and only if there exists $\beta\in k^{\ast}$ such that $b=\beta b^{'}$, $c=\beta c^{'}$,
$d=\beta^2 d^{'}$ and $e=\beta e^{'}$. Obviously, $A_{0,0,d,e}$ is not isomorphic to $A_{b^{'},c^{'},d^{'},e^{'}}$ when one of $b^{'}$ and $c^{'}$ is not zero. And, $A_{0,c,d,e}$ and $A_{b,0,d,e}$ are not isomorphic to $A_{b^{'},c^{'},d^{'},e^{'}}$ for
$b^{'}\neq 0$, $c^{'}\neq 0$ and any $c$, $d$, $e$, $c^{'}$, $d^{'}$, $e^{'}\in k$.

Thus, we can classify these left-symmetric algebras into four
cases, i.e., $A_{0,0,d,e}$, $A_{b,0,d,b}$ $(b\neq 0)$,
$A_{0,c,d,c}$ $(c\neq 0)$ and $A_{b,b,d,b}$ $(b\neq 0)$.
By the above discussion, any two left-symmetric algebras in different cases are not isomorphic.
Moreover, $A_{0,0,d,e}$ and $A_{0,0,d^{'},e^{'}}$
are isomorphic if and only if there exists $\beta\in k^{\ast}$ such that $e=\beta e^{'}$, 
and $d=\beta^2 d^{'}$; $A_{b,0,d,b}$$(b\neq 0)$ and $A_{b^{'},0,d^{'},b^{'}}$ $(b^{'}\neq 0)$ 
are isomorphic if and only if there exists $\beta\in k^{\ast}$ such that $b=\beta b^{'}$,
and $d=\beta^2 d^{'}$; $A_{0,c,d,c}$ $(c\neq 0)$ and $A_{0,c^{'},d^{'},c^{'}}$ $(c^{'}\neq 0)$
are isomorphic if and only if there exists $\beta\in k^{\ast}$ such that $c=\beta c^{'}$,
and $d=\beta^2 d^{'}$; $A_{b,b,d,b}$ $(b\neq 0)$ and $A_{b^{'},b^{'},d^{'},b^{'}}$ $(b^{'}\neq 0)$
are isomorphic if and only if there exists $\beta\in k^{\ast}$ such that $b=\beta b^{'}$,
and $d=\beta^2 d^{'}$.

Thus, $\mathcal{H}_A^2(k,A)=((k\times k)/\equiv_1)\cup ((k^\ast \times k)/\equiv_1 )\cup ((k^\ast \times k)/\equiv_1 )\cup (k^\ast \times k/\equiv_1)$ where $\equiv_1$ is the
following relation on $k \times k$ or $k^\ast \times k$: $(u,v)\equiv_1 (u^{'},v^{'})$ if and only if
there exists $\beta\in k^{\ast}$ such that $u=\beta u^{'}$,
and $v=\beta^2 v^{'}$.

Moreover, by Theorem \ref{tt1}, we can also get $\mathcal{H}^2(k,A)=(k \times k) \cup (k^\ast \times k)\cup (k^\ast \times k)\cup (k^\ast \times k)$.
\end{example}

\begin{example}
Let $k$ be a field of characteristic zero and $A$ be a 3-dimensional Novikov algebra with a basis
$\{e_1,e_2,e_3\}$ and the non-zero products given by
\begin{eqnarray}
\label{qw5}e_1\circ e_1=-e_1+e_2,~~e_2\circ e_1=-e_2,~~e_3\circ e_1=-e_3.
\end{eqnarray}
This Novikov algebra whose sub-adjacent Lie algebra is just $\mathfrak{g}(A)$ with the Lie brackets given by
$[e_1,e_2]=e_2$, $[e_1,e_3]=e_3$ can be referred to \cite{Dd}.

By a long but straightforward computation, we obtain the following three flag datums $(h,g,D,T,a_0,\alpha)$:\\
{\bf Case 1} \begin{eqnarray*}
&&h=g=0,~~a_0=c_1 e_1+c_2 e_2+c_3e_3,\\
&&D=\left(
    \begin{array}{ccc}
      a_{11} & a_{12} & a_{13} \\
      0 &0& 0  \\
      0 & 0 & 0  \\
    \end{array}
  \right),~~~T=\left(
    \begin{array}{ccc}
      a_{11} & -a_{11}& 0  \\
      0 & a_{11} & 0  \\
      0 & 0 & a_{11}\\
    \end{array}
  \right),
\end{eqnarray*}
where $a_{11}$, $a_{12}$, $a_{13}$, $c_1$, $c_2$, $c_3\in k$ and $c_1=a_{11}\alpha -a_{11}^2$,
$c_2=\alpha(a_{11}+a_{12})-a_{11}a_{12}$, $c_3=\alpha a_{13}-a_{11}a_{13}$.

In this case, any 4-dimensional Novikov algebra that contains $A$ as a Novikov subalgebra
is isomorphic to the following Novikov algebra denoted by $A_{a_{11},a_{12},a_{13},\alpha}$ with the basis
$\{e_1,e_2,e_3,x\}$ and the products given by (\ref{qw5}) and
\begin{eqnarray*}
&&e_1\circ x=a_{11}(e_1-e_2),~~x\circ e_1=a_{11}e_1+a_{12}e_2+a_{13}e_3,\\
&&e_2\circ x=a_{11}e_2,~~x\circ e_2=0,~~e_3\circ x=a_{11}e_3,~~x\circ e_3=0,\\
&&x\circ x=(a_{11}\alpha -a_{11}^2)e_1+(\alpha(a_{11}+a_{12})-a_{11}a_{12})e_2+(\alpha a_{13}-a_{11}a_{13})e_3+\alpha x,
\end{eqnarray*}
where $a_{11}$, $a_{12}$, $a_{13}$, $\alpha\in k$.

By Theorem \ref{tt1}, two such Novikov algebras $A_{a_{11},a_{12},a_{13},\alpha}$ and $A_{a_{11}^{'},a_{12}^{'},a_{13}^{'},\alpha^{'}}$
are isomorphic if and only if there exist $\beta\in k^{\ast}$ and $d_1$, $d_2$, $d_3\in k$ such that $\alpha=\beta \alpha^{'}$, $a_{11}=-d_1+\beta a_{11}^{'}$, $a_{12}=d_1-d_2+\beta a_{12}^{'}$, $a_{13}=-d_3+\beta a_{13}^{'}$,
$c_1=-d_1^2+2\beta d_1 a_{11}^{'}+\beta^2c_1^{'}-\alpha d_1$, $c_2=d_1^2-d_1d_2+\beta d_1a_{12}^{'}-\beta d_1a_{11}^{'}+\beta d_2 a_{11}^{'}+\beta^2c_2^{'}-\alpha d_2$
and $c_3=-d_1 d_3+\beta d_1 a_{13}^{'}+\beta d_3 a_{11}^{'}+\beta^2c_3^{'}-\alpha d_3$.\\

{\bf Case 2}\begin{eqnarray*}
&&h(e_1)=h(e_2)=h(e_3)=0,\\
&&g(e_1)=-1,~~g(e_2)=g(e_3)=0,\\
&&D=\left(
    \begin{array}{ccc}
      0 & a_{12} & 0 \\
      0 &0& 0  \\
      0 & 0 & 0  \\
    \end{array}
  \right),~~~T=\left(
    \begin{array}{ccc}
      -a_{12} & b_{12}& b_{13}  \\
      0 & -a_{12} & 0  \\
      0 & -a_{12} & 0\\
    \end{array}
  \right),\\
&&a_0=(a_{12}b_{12}+a_{12}b_{13})e_2,~~~\alpha=-a_{12},
\end{eqnarray*}
where $a_{12}$, $b_{12}$, $b_{13}\in k$.

In this case, any 4-dimensional Novikov algebra that contains $A$ as a Novikov subalgebra
is isomorphic to the following Novikov algebra denoted by $A_{a_{11},a_{12},a_{13},\alpha}$ with the basis
$\{e_1,e_2,e_3,x\}$ and the products given by (\ref{qw5}) and
\begin{eqnarray*}
&&e_1\circ x=-a_{12}e_1+b_{12}e_2+b_{13}e_3,~~e_2\circ x=-a_{12}e_2,~~e_3\circ x=-a_{12}e_2,\\
&&x\circ e_1=a_{12}e_2-x,~~x\circ e_2=x\circ e_3=0,\\
&&x\circ x=-a_{12}x+(a_{12}b_{12}+a_{12}b_{13})e_2,
\end{eqnarray*}
where $a_{12}$, $b_{12}$, $b_{13}\in k$.

By Theorem \ref{tt1}, two such Novikov algebras $A_{a_{12},b_{12},b_{13}}$ and $A_{a_{12}^{'},b_{12}^{'},b_{13}^{'}}$
are isomorphic if and only if there exist $\beta\in k^{\ast}$ and $d_1$, $d_2$, $d_3\in k$ such that $a_{12}=d_1+\beta a_{12}^{'}$,
$b_{12}=\beta b_{12}^{'}+d_1$, $b_{13}=\beta b_{13}^{'}$,
$a_{12}b_{12}+a_{12}b_{13}=d_1^2-d_1d_2+\beta d_1 a_{12}^{'}+\beta d_1 b_{12}^{'}-\beta d_2 a_{12}^{'}-\beta d_3 a_{12}^{'}
+\beta^2(a_{12}^{'}b_{12}^{'}+a_{12}^{'}b_{13}^{'})+a_{12}d_2$, and $\beta d_1 b_{13}^{'}+a_{12}d_3-d_1d_3=0$.

{\bf Case 3} \begin{eqnarray*}
&&h(e_1)=-1,~~h(e_2)=h(e_3)=0,\\
&&g(e_1)=-1,~~g(e_2)=g(e_3)=0,\\
&&D=\left(
    \begin{array}{ccc}
      0 & a_{12} & a_{13} \\
      0 &0& 0  \\
      0 & 0 & 0  \\
    \end{array}
  \right),~~~T=\left(
    \begin{array}{ccc}
      0 & b_{12}& b_{13}  \\
      0 & -a_{12} & -a_{13}  \\
      0 & b_{32} & b_{33}\\
    \end{array}
  \right),\\
&&a_0=c_1e_1+c_2e_2+c_3e_3,
\end{eqnarray*}
where $a_{12}$, $a_{13}$, $b_{12}$, $b_{13}$, $b_{32}$, $b_{33}$, $c_1$, $c_2$, $c_3\in k$. In this case, there are two subcases. \\
{\bf Subcase 1}: When $a_{13}\neq 0$ or $b_{32}\neq 0$, $\alpha=b_{33}-a_{12}$, $c_1=a_{13}b_{32}-a_{12}b_{33}$, $c_2=a_{12}^2-a_{13}b_{32}+a_{12}b_{33}-b_{33} b_{12}$,
and $c_3=a_{12}a_{13}-a_{13}b_{12}-a_{13}b_{33}+a_{12}b_{13}$.\\
{\bf Subcase 2}: When $a_{13}=b_{32}=0$, $c_1=-a_{12}\alpha-a_{12}^2=\alpha b_{33}-b_{33}^2$, $c_2=2a_{12}^2-b_{12}a_{12}+a_{12}\alpha-\alpha b_{12}$
and $c_3=b_{13}b_{33}-\alpha b_{13}$.

In this case, any 4-dimensional Novikov algebra that contains $A$ as a Novikov subalgebra
is isomorphic to the following Novikov algebra $A_{a_{11},a_{12},b_{12},b_{13}, b_{32}, b_{33},c_1, c_2, c_3,\alpha}$ with the basis
$\{e_1,e_2,e_3,x\}$ and the products given by (\ref{qw5}) and
\begin{eqnarray}
\label{ne1}e_1\circ x=-x+b_{12}e_2+b_{13}e_3,~~e_2\circ x=-a_{12}e_2-a_{13}e_3,\\
\label{ne2}~~e_3\circ x=b_{32}e_2+b_{33}e_3,~~x\circ e_1=a_{12}e_2+a_{13}e_3-x,\\
\label{ne3}~~x\circ e_2=x\circ e_3=0,~~x\circ x=\alpha x+c_1e_1+c_2e_2+c_3e_3,
\end{eqnarray}
where these coefficients satisfy the conditions in Subcase 1 and Subcase 2.

By Theorem \ref{tt1}, two such Novikov algebras $A_{a_{12},a_{13},b_{12},b_{13}, b_{32}, b_{33},c_1, c_2, c_3,\alpha}$  and $A_{a_{12}^{'},a_{13}^{'},b_{12}^{'},b_{13}^{'}, b_{32}^{'}, b_{33}^{'},c_1^{'}, c_2^{'}, c_3^{'},\alpha^{'}}$
are isomorphic if and only if there exist $\beta\in k^{\ast}$ and $d_1$, $d_2$, $d_3\in k$ such that
\begin{eqnarray}
\label{ne3}\alpha =\beta \alpha^{'}-2d_1,~~a_{12}=d_1+\beta a_{12}^{'},~~a_{13}=\beta a_{13}^{'},\\
\label{ne4}b_{12}=d_1+d_2+\beta b_{12}^{'},~~b_{13}=\beta b_{13}^{'}+d_3,~~b_{32}=\beta b_{32}^{'},\\
\label{ne5}b_{33}=-d_1+\beta b_{33}^{'},~~c_1=-d_1^2+c_1^{'}\beta^2-\alpha d_1,\\
\label{ne6}c_2=d_1^2-d_1d_2+\beta d_1a_{12}^{'}+\beta d_1b_{12}^{'}-\beta d_2a_{12}^{'}+\beta d_3 b_{32}^{'}+\beta^2c_2^{'}-\alpha d_2,\\
\label{ne7}c_3=-d_1d_3+\beta d_1a_{13}^{'}+\beta d_1b_{13}^{'}-\beta d_2a_{13}^{'}+\beta d_3 b_{33}^{'}+\beta^2c_3^{'}-\alpha d_3.
\end{eqnarray}

In detail, denote the Novikov algebras corresponding to Subcase 1 and Subcase 2 by $A_{a_{12},b_{12},b_{13}, b_{32}, b_{33}}$
and $A_{a_{12},b_{12},b_{13}, b_{33},\alpha}$ respectively. By (\ref{ne3}) and (\ref{ne4}), it is easy to see that $A_{a_{12},b_{12},b_{13}, b_{32}, b_{33}}$
is not isomorphic to any $A_{a_{12}^{'},b_{12}^{'},b_{13}^{'}, b_{33}^{'},\alpha^{'}}$.
And $A_{a_{12},b_{12},b_{13}, b_{32}, b_{33}}$ is isomporphic to $A_{a_{12}^{'},b_{12}^{'},b_{13}^{'}, b_{32}^{'}, b_{33}^{'}}$ if and only if
there exist $\beta\in k^{\ast}$ and $d_1$, $d_2$, $d_3\in k$ such that (\ref{ne3})-(\ref{ne7}) hold where $\alpha=b_{33}-a_{12}$, $c_1=a_{13}b_{32}-a_{12}b_{33}$, $c_2=a_{12}^2-a_{13}b_{32}+a_{12}b_{33}-b_{33} b_{12}$, $c_3=a_{12}a_{13}-a_{13}b_{12}-a_{13}b_{33}+a_{12}b_{13}$,$\alpha^{'}=b_{33}^{'}-a_{12}^{'}$, $c_1^{'}=a_{13}^{'}b_{32}^{'}-a_{12}^{'}b_{33}^{'}$, $c_2^{'}={a_{12}^{'}}^2-a_{13}^{'}b_{32}^{'}+a_{12}^{'}b_{33}^{'}-b_{33}^{'}b_{12}^{'}$
and $c_3^{'}=a_{12}^{'}a_{13}^{'}-a_{13}^{'}b_{12}^{'}-a_{13}^{'}b_{33}^{'}+a_{12}^{'}b_{13}^{'}$.
Similarly, $A_{a_{12},b_{12},b_{13}, b_{33},\alpha}$ is isomorphic to $A_{a_{12}^{'},b_{12}^{'},b_{13}^{'}, b_{33}^{'},\alpha^{'}}$
if and only if
there exist $\beta\in k^{\ast}$ and $d_1$, $d_2$, $d_3\in k$ such that (\ref{ne3})-(\ref{ne7}) hold where $a_{13}=a_{13}^{'}=b_{32}=b_{32}^{'}=0$, $c_1=-a_{12}\alpha-a_{12}^2=\alpha b_{33}-b_{33}^2$, $c_2=2a_{12}^2-b_{12}a_{12}+a_{12}\alpha-\alpha b_{12}$
, $c_3=b_{13}b_{33}-\alpha b_{13}$, $c_1^{'}=-a_{12}^{'}\alpha^{'}-{a_{12}^{'}}^2=\alpha^{'} b_{33}^{'}-{b_{33}^{'}}^2$, $c_2^{'}=2{a_{12}^{'}}^2-b_{12}^{'}a_{12}^{'}+a_{12}^{'}\alpha^{'}-\alpha^{'} b_{12}^{'}$
, $c_3^{'}=b_{13}^{'}b_{33}^{'}-\alpha^{'}b_{13}^{'}$.

{\bf Case 4}\begin{eqnarray*}
&&h(e_1)=\gamma,~~h(e_2)=h(e_3)=0,\\
&&g(e_1)=-1,~~g(e_2)=g(e_3)=0,\\
&&D=\left(
    \begin{array}{ccc}
      0 & 0 & 0 \\
      0 &0& 0  \\
      0 & 0 & 0  \\
    \end{array}
  \right),~~~T=\left(
    \begin{array}{ccc}
      0 & b_{12}& b_{13}  \\
      0 & 0 & 0  \\
      0 & 0 & 0\\
    \end{array}
  \right),\\
&&a_0=c_2e_2+c_3e_3,~~~\alpha=0,
\end{eqnarray*}
where $c_2$, $c_3$, $\gamma \in k$, $\gamma \neq 0$ and $\gamma \neq -1$ and when $\gamma\neq-\frac{1}{2}$, $c_2=c_3=0$.

Therefore, in this case, there are two classes of Novikov algebras containing $A$ as a subalgebra.
One is the Novikov algebra $A_{b_{12},b_{13},c_2,c_3}$  with the basis
$\{e_1,e_2,e_3,x\}$ and the products given by (\ref{qw5}) and
\begin{eqnarray}
&&e_1\circ x=-\frac{1}{2}x+b_{12}e_2+b_{13}e_3,~~x\circ e_1=-x,\\
&&e_2\circ x=x\circ e_2=e_3\circ x=x\circ e_3=0,\\
&&x\circ x=c_2e_2+c_3e_3,
\end{eqnarray}
where $b_{12}$, $b_{13}$, $c_2$, $c_3\in k$. By Theorem \ref{tt1}, two such Novikov algebras $A_{b_{12},b_{13},c_2,c_3}$   and $A_{b_{12}^{'},b_{13}^{'},c_2^{'},c_3^{'}}$
are isomorphic if and only if there exist $\beta\in k^{\ast}$  and $d_2$, $d_3\in k$ such that
$b_{12}=\beta b_{12}^{'}+\frac{1}{2}d_2$, $b_{13}=\beta b_{13}^{'}+\frac{1}{2}d_3$, $c_2=\beta^2 c_2^{'}$ and
$c_3=\beta^2 c_3^{'}$.

The other is the Novikov algebra $A_{b_{12},b_{13},\gamma}$  with the basis
$\{e_1,e_2,e_3,x\}$ and the products given by (\ref{qw5}) and
\begin{eqnarray}
&&e_1\circ x=\gamma x+b_{12}e_2+b_{13}e_3,~~x\circ e_1=-x,\\
&&e_2\circ x=x\circ e_2=e_3\circ x=x\circ e_3=x\circ x=0,
\end{eqnarray}
where $b_{12}$, $b_{13}$, $\gamma\in k$ and $\gamma\neq 0$, $\gamma\neq -1$, and $\gamma\neq -\frac{1}{2}$.
By Theorem \ref{tt1}, two such Novikov algebras $A_{b_{12},b_{13},\gamma}$   and $A_{b_{12}^{'},b_{13}^{'},\gamma^{'}}$
are isomorphic if and only if there exist $\beta\in k^{\ast}$  and $d_2$, $d_3\in k$ such that
$b_{12}=\beta b_{12}^{'}-\gamma d_2$, $b_{13}=\beta b_{13}^{'}-\gamma d_3$. Therefore, in this case, if $\gamma$ is fixed, all such Novikov algebras are isomorphic.

It is known that if $(h,g,D,T,a_0, \alpha)$ $\equiv$ $(h^{'},g^{'},D^{'},T^{'},a^{'}_0, \alpha^{'})$,
$h=h^{'}$, $g=g^{'}$. Therefore, the Novikov algebras in different cases are not isomorphic.
Thus, the classifying object $\text{Extd}(k^4, A)\cong \mathcal{H}_A^2(k,A)$ has been described: it is equal to the disjoint union
of the six quotient spaces described above.
\end{example}
\section{Classifying complements for left-symmetric algebras}
In this section, we will study the classifying complements problem for left-symmetric (resp. Novikov) algebras.

Let $A\subseteq E$ be a left-symmetric (resp. Novikov) subalgebra of $(E,\circ)$. A left-symmetric (resp. Novikov) subalgebra $B$ of $(E,\circ)$ is called a
\emph{complement} of $A$ in $(E,\circ)$ (or an $A$-complement of $(E,\circ)$) if $E=A+B$ and $A\cap B=\{0\}$. If
$B$ is an $A$-complement in $(E,\circ)$, by Remark \ref{rr1}, we get $E\cong A\bowtie B$ for some bicrossed product of $A$ and
$B$ (the detail construction can be referred to the proof of Theorem \ref{pt1} or \cite{Bai}).

For a left-symmetric (resp. Novikov) subalgebra $A$ of $(E,\circ)$, denote $\mathcal{F}(A,E)$ the set of the isomorphism classes of all $A$-complements in $E$.
Define the factorization index of $A$ in $E$ as $[E:A]:=|\mathcal{F}(A,E)|$.

\begin{definition}
Let $(A, B,l_A, r_A, l_B, r_B)$ be a matched pair of left-symmetric (resp. Novikov) algebras. A linear map
$\varphi: B\rightarrow A$ is called a \emph{deformation map} of the matched pair $(A, B,l_A, r_A, l_B, r_B)$ if $\varphi$
satisfies the following condition for any $x$, $y\in B$:
\begin{gather}
\varphi(x\circ y)-\varphi(x)\circ \varphi(y)=l_B(x)\varphi(y)+r_B(y)\varphi(x)\\
-\varphi(l_A(\varphi(x))y+r_A(\varphi(y))x).\nonumber
\end{gather}

Denote by the set of all deformation maps of the matched pair $(A, B,l_A,r_A,l_B,r_B)$ by $\mathcal{DM}(B,A|(l_A,r_A,l_B,r_B))$.
\end{definition}

Next, we begin to study the classifying complements problem for left-symmetric (resp. Novikov) algebras using the concept of deformation map.

\begin{theorem}\label{th2}
Let $A$ be a left-symmetric (resp. Novikov) subalgebra of $(E,\circ)$, $B$ a given $A$-complement of $E$ with the associated canonical matched pair
$(A,B,l_A,r_A,$ $l_B,r_B)$.\\

(1) Let $\varphi: B\rightarrow A$ be a deformation map of the above matched pair. Then,
$B_{\varphi}:=B$ as a vector space is a left-symmetric (resp. Novikov) algebra with the new product given as follows for any $x$, $y\in B$:
\begin{eqnarray}\label{f1}
x\circ_{\varphi} y:=x\circ y+l_A(\varphi(x))y+r_A(\varphi(y))x.
\end{eqnarray}
$B_{\varphi}$ is called the $\varphi$-deformation of $B$. Moreover, $B_{\varphi}$ is an $A$-complement of $E$.\\

(2) $\overline{B}$ is an $A$-complement of $E$ if and only if $\overline{B}$ is isomorphic to $B_\varphi $ for some deformation map
$\varphi: B\rightarrow A$ of the matched pair $(A,B,l_A,r_A,l_B,r_B)$.
\end{theorem}
\begin{proof}
(1) Given a deformation map $\varphi: B\rightarrow A$. Let
$f_\varphi: B\rightarrow E=A\bowtie B$ be the linear map defined as $ f_\varphi(x)=(\varphi(x),x)$ for any $x\in B$.

By the definition of the deformation map, we get that for any $x$, $y\in B$,
\begin{eqnarray*}
&&[(\varphi(x),x),(\varphi(y),y)]\\
&=&(\varphi(x)\circ \varphi(y)+l_B(y)\varphi(x)+l_B(x)\varphi(y),x\circ y+l_A(\varphi(x))y+r_A(\varphi(y))x)\\
&=&(\varphi(x\circ y+l_A(\varphi(x))y+r_A(\varphi(y))x), x\circ y+l_A(\varphi(x))y+r_A(\varphi(y))x).
\end{eqnarray*}
Therefore, $[(\varphi(x),x),(\varphi(y),y)]\in \text{Im}(f_\varphi)=\{(\varphi(x),x)\mid x\in B\}$. Thus, $\text{Im}(f_\varphi)$ is
a left-symmetric (resp. Novikov) subalgebra of $E=A\bowtie B$. Here, $A\cong A\times \{0\}$ is viewed as a subalgebra of $A\bowtie B$. It is easy to see that $A\cap \text{Im}(f_\varphi)=\{0\}$ and $(a,x)=(a-\varphi(x),0)+(\varphi(x),x)\in A+B$ for all $a\in A$, $x\in B$. Hence,
$\text{Im}(f_\varphi)$ is an $A$-complement of $E=A\bowtie B$. Then, we only need to prove that
$B_\varphi $ and $\text{Im}(f_\varphi)$ are isomorphic as algebras. Denote by $\widetilde{f_\varphi}:
B\rightarrow \text{Im}(f_\varphi)$ the linear map induced by $f_\varphi$. Obviously, $\widetilde{f_\varphi}$ is a linear isomorphism.
Next, we prove that $\widetilde{f_\varphi}$ is also an algebra morphism if the product of $B$ is given by (\ref{f1}).
For any $x$, $y\in B$, we get
\begin{eqnarray*}
&&\widetilde{f_\varphi}(x\circ_\varphi y)\\
&=& \widetilde{f_\varphi}(x\circ y+l_A(\varphi(x))y+r_A(\varphi(y))x)\\
&=&(\varphi(x\circ y+l_A(\varphi(x))y+r_A(\varphi(y))x),x\circ y+l_A(\varphi(x))y+r_A(\varphi(y))x)\\
&=&(\varphi(x)\circ \varphi(y)+l_B(x)\varphi(y)+r_B(y)\varphi(x), x\circ y+l_A(\varphi(x))y+r_A(\varphi(y))x)\\
&=&[(\varphi(x),x),(\varphi(y),y)]\\
&=&[\widetilde{f_\varphi}(x), \widetilde{f_\varphi}(y)].
\end{eqnarray*}
Thus, $B_\varphi$ is a left-symmetric (resp. Novikov) algebra.

(2) Let $\overline{B}$ be an arbitrary of $A$-complement of $E$. Since $E=A\oplus B=A\oplus \overline{B}$, we can obtain
four linear maps:
$$u: B\rightarrow A,~~~~v:B\rightarrow \overline{B},~~~s: \overline{B}\rightarrow A,~~~t: \overline{B}\rightarrow B$$
such that for all $x\in B$ and $y\in \overline{B}$, we get
\begin{eqnarray}\label{f2}
x=u(x)+v(x),~~~~y=s(y)+t(y).
\end{eqnarray}
Thus, by (\ref{f2}), for $x=t(y)$, $y\in \overline{B}$, we obtain
\begin{eqnarray}
-s(y)+y=t(y)=u(t(y))+v(t(y)).
\end{eqnarray}
By the unique decomposition in a direct sum, we get $u(t(y))=-s(y)$ and $v(t(y))=y$ for all $y\in \overline{B}$. Thus,
$v$ is surjective. If there exists some non-zero element $x\in B$ such that $x=u(x)$, then it contradicts with
$A\cap B=\{0\}$. Therefore, $v$ is injective. Hence, $v: B\rightarrow \overline{B}$ is a linear isomorphism of vector spaces.
Denote by $\widetilde{v}: B\rightarrow A\bowtie B$ the composition $\widetilde{v}: B\rightarrow \overline{B}\hookrightarrow^i E=A\bowtie B$.
Thus, $\widetilde{v}(x)=(-u(x),x)$ for all $x\in B$. Let $\varphi=-u$. Then, by the fact that $\overline{B}=\text{Im}(v)=\text{Im}(\widetilde{v})$ is
a left-symmetric (resp. Novikov) subalgebra of $E=A\bowtie B$, we get
\begin{eqnarray*}
&&[(\varphi(x),x),(\varphi(y),y)]\\
&=&(\varphi(x)\circ \varphi(y)+l_B(y)\varphi(x)+l_B(x)\varphi(y),x\circ y+l_A(\varphi(x))y+r_A(\varphi(y))x)\\
&=&(\varphi(z),z)
\end{eqnarray*}
for some $z\in B$. Here, $z=x\circ y+l_A(\varphi(x))y+r_A(\varphi(y))x$. Thus, $\varphi$ is a deformation map of the matched pair $(A,B,l_A,r_A,l_B,r_B)$ and $v:B_\varphi\rightarrow\overline{ B }$ is also a left-symmetric (resp. Novikov)  algebra morphism. Therefore, $\overline{B}\cong B_\varphi$.
\end{proof}

\begin{definition}
Let $(A,B,l_A,r_A,l_B,r_B)$ be a matched pair of left-symmetric (resp. Novikov) algebras. For two deformation maps
$\varphi$, $\psi: B\rightarrow A$, if there exists $\rho: B\rightarrow B$ a linear automorphism of $B$ such that
for any $x$, $y\in B$:
\begin{gather}
\rho(x\circ y)-\rho(x)\circ \rho(y)=l_A(\psi(\rho(x)))\rho(y)+r_A(\psi(\rho(y))\rho(x)\\
-\rho(l_A(\varphi(x))y)-\rho(r_A(\varphi(y)x)),\nonumber
\end{gather}
$\varphi$ and $\psi$ are called \emph{equivalent}. Denote it by $\varphi\sim \psi$.
\end{definition}

\begin{theorem}
Let $A$ be a left-symmetric (resp. Novikov) subalgebra of $E$, $B$ an $A$-complement of $E$ and
$(A, B, l_A, r_A, l_B, r_B)$ the associated matched pair. Then, $\sim $ is an equivalence relation on
the set $\mathcal{DM}(B,A|(l_A,r_A,l_B,r_B))$ and the map
\begin{eqnarray*}
\mathcal{HC}^2(B,A|(l_A,r_A,l_B,r_B)):=\mathcal{DM}(B,A|(l_A,r_A,l_B,r_B))/\sim \rightarrow \mathcal{F}(A,B), ~~~\overline{\varphi}\mapsto B_\varphi,
\end{eqnarray*}
is a bijection between $\mathcal{HC}^2(B,A|(l_A,r_A,l_B,r_B))$ and the isomorphism classes of all $A$-complements of $E$.
In particular, $[E:A]=|\mathcal{HC}^2(B,A|(l_A,r_A,l_B,r_B))|$.
\end{theorem}
\begin{proof}
It is easy to see that two deformation maps $\varphi$ and $\psi$ are equivalent if and only if the corresponding
left-symmetric (resp. Novikov) algebras $B_\varphi$ and $B_\psi$ are isomorphic. Then, by Theorem \ref{th2}, we obtain
this theorem.
\end{proof}

Finally, we provide a explicit example to compute $[E:A]$.
\begin{example}
Let $k$ be the field of real numbers $\mathbb{R}$ or the field of complex numbers $\mathbb{C}$.
Let $E$ be the 4-dimensional left-symmetric algebra with basis $\{e_1,e_2,e_3,e_4\}$ and the product
defined as follows (other products vanishing):
\begin{eqnarray}
e_1\circ e_3=e_3,~~e_2\circ e_2=2e_2,~~e_3\circ e_4=e_2,\\
e_1\circ e_4=-e_4,~~e_2\circ e_3=e_3,~~e_4\circ e_3=e_2,~~e_2\circ e_4=e_4.
\end{eqnarray}
This algebra can be referred to Example 3.34 in \cite{Bu1}.
Assume that $A$ and $B$ are the left-symmetric subalgebras of $E$ with basis $\{e_1,e_3\}$ and $\{e_2,e_4\}$ respectively.
Obviously, $B$ is a complement of $A$ in $E$. Therefore, $E$ is isomorphic to the bicrossed product corresponding to
the matched pair $(A,B,l_A,r_A,l_B,r_B)$ whose actions are given by (other actions vanishing)
\begin{eqnarray*}
l_A(e_1)e_4=-e_4,~~l_A(e_3)e_4=e_2,~~r_A(e_3)e_4=e_2,~~l_B(e_2)e_3=e_3.
\end{eqnarray*}

By some simple computations, we can obtain that any deformation map associated with the above matched pair of left-symmetric
algebras is of the following form:
\begin{eqnarray*}
\varphi_b: B\rightarrow A,~~~~~~~~~~~\varphi(e_2)=0,~~\varphi(e_4)=be_3,
\end{eqnarray*}
for some $b\in k$. Furthermore, the $\varphi_b$-deformation of $B$ has the following product:
\begin{eqnarray*}
e_2\circ_{\varphi_b}e_2=2e_2,~~e_2\circ_{\varphi_b}e_4=e_4,\\
e_4\circ_{\varphi_b}e_2=0,~~e_4\circ_{\varphi_b} e_4=2be_2.
\end{eqnarray*}

Next, we consider $[E:A]$ in two cases.\\

Assume $k=\mathbb{C}$. If $b=0$, $B_{\varphi_b}=B$. For $b\neq 0$, the left-symmetric algebra $B_{\varphi_b}$ is isomorphic
to $B_{\varphi_1}$ with the isomorphism $\rho: B_{\varphi_b}\rightarrow B_{\varphi_1}$ given by
$\rho(e_2)=e_2,$ $\rho(e_4)=\sqrt{b}e_4$. Moreover, it is obvious that $B_{\varphi_1}$ is not isomorphic to $B$.
Therefore, in this case, $[E:A]=2$.

Assume $k=\mathbb{R}$. If $b=0$, $B_{\varphi_b}=B$. If $b> 0$, the left-symmetric algebra $B_{\varphi_b}$ is isomorphic
to $B_{\varphi_1}$ with the isomorphism $\rho: B_{\varphi_b}\rightarrow B_{\varphi_1}$ given by
$\rho(e_2)=e_2,$ $\rho(e_4)=\sqrt{b}e_4$. Similarly, if $b< 0$, the left-symmetric algebra $B_{\varphi_b}$ is isomorphic
to $B_{\varphi_{-1}}$ with the isomorphism $\rho: B_{\varphi_b}\rightarrow B_{\varphi_{-1}}$ given by
$\rho(e_2)=e_2,$ $\rho(e_4)=\sqrt{-b}e_4$. Moreover, it can be checked that $B_{\varphi_1}$, $B$ and $B_{\varphi_{-1}}$
are not isomorphic to each other.
Therefore, in this case, $[E:A]=3$.

\end{example}

\end{document}